\documentclass{svjour3}
\usepackage{amssymb}
\usepackage{multirow}
\usepackage{rotating}
\usepackage{cite}
\usepackage{lscape}
\usepackage{caption}
\usepackage{amsmath}
\usepackage{graphicx}
\usepackage{color}
\newcommand{\eTRS}{\rm (eTRS)\,}
\newcommand{\meTRS}{\rm (m-eTRS)\,}
\newcommand{\TTRS}{{\rm TTRS} }
\newcommand{\TTRSp}{{\rm TTRS}}

\newcommand{\TRS}{{\rm TRS }}
\newcommand{\TRSp}{{\rm TRS}}

\newcommand{\LNGM}{\rm LNGM\,}

\DeclareMathOperator{\trace}{trace}

\begin{document}

\title{A hybrid algorithm for the two-trust-region subproblem}

\author{Saeid Ansary Karbasy, Maziar Salahi}%      \and
 %       Maziar Salahi %etc.
	%}

\authorrunning{Two-trust-region subproblem} % if too long for running head

\institute{Saeid Ansary Karbasy \and Maziar Salahi\at
             Faculty of Mathematical Sciences, University of Guilan, Rasht, Iran \\
            Tel.: +98-131-3233901\\
           Fax: +98-131-3233509\\
          \email{saeidansary144@yahoo.com, salahim@guilan.ac.ir}           %  \\
     %        \emph{Present address:} of F. Author  %  if needed
          }

\date{}

\maketitle

\begin{abstract}
Two-trust-region subproblem (TTRS), which is the minimization of a general quadratic function over the intersection of two full-dimensional ellipsoids,  has been the subject of several recent research.  In this paper,  to solve TTRS, a hybrid of efficient algorithms for finding global and local-nonglobal minimizers of trust-region subproblem and the  alternating direction method of multipliers (ADMM) is proposed. The   convergence  of the ADMM steps to the first order stationary condition is proved under certain conditions. On several classes of test problems, we  compare the new algorithm with  the recent algorithm of Sakaue et. al's  \cite{SakaueNakat:16} and Snopt software.
%For small scale problems, the ADMM algorithm is compared with the algorithm of  Sakaue et. al's  \cite{SakaueNakat:16} that assures finding a global solution in polynomial time.  Moreover, on  medium and large scale  problem ADMM is  compared with the Snopt.
\end{abstract}\\


%\tableofcontents
\section{Introduction}
This paper studies the two-trust-region subproblem
{\rm(TTRS)}, which is the minimization of a general quadratic function over the intersection
of two full-dimensional ellipsoids:
%Consider the following two trust region subproblem   \TTRS:
\begin{align}\label{mrs}
\min \quad &\frac12x^TAx +a^Tx \notag \\
&||x||^2 \leq \delta_1^2 , \tag{\rm {TTRS}}\\
&(x-c)^TB(x-c) \leq \delta_2^2 , \notag
\end{align}
where $A\in \mathbb{R}^{n\times n}$ is a symmetric matrix,  $B \in \mathbb{R}^{n\times n}$ is  symmetric positive definite,
   $c \in \mathbb{R}^n$ and   $\delta_1,\delta_2 \in\mathbb{R}$. If $A$ is positive semidefinite, then \TTRS is solvable in polynomial time
 by second-order cone programming. Therefore, throughout this paper we assume $A$ is  indefinite.

When $B=0$,     \TTRS reduces to the well-known trust-region subproblem (TRS)
that has been widely studied and efficient algorithms exist to solve it \cite{Adachi:15,con}.
 However, the additional constraint makes   \TTRS more challenging than  \TRSp.
%%%%%%%%%%%%%%%%%%%%%%%%%%%%
%%%%%%%%%%%%%%%%%%%%%%%%%%%%%%%%%%%%%%
 \TTRS is originally introduced by Celis, Dennis, and Tapia \cite{CelisDT:84} and   called the CDT subproblem.
Two algorithms have been suggested   for the CDT subproblem under
the assumption that the objective function is convex \cite{Yuan:91, Heink:94}.
Zhang proposed  an algorithm for the CDT subproblem under the assumption that
the optimal  Lagrangian Hessian  is positive semidefinite \cite{Zhang:92}.
However, Yuan proved    that the Hessian of  Lagrangian for  CDT subproblem may have one negative eigenvalue at   global solution \cite{Yuan:90}.  In 1994,  Mart\'inez proved the existence of at most one  local non-global minimum (LNGM) for \TRSp, which in the case of feasibility for TTRS is a candidate for its optimal solution \cite{Mart:94}.   Peng and  Yuan showed  that the  CDT  subproblem can have a duality gap and studied its necessary and sufficient optimality conditions \cite{Peng:97}.  Later in 2000, Nesterov  and Wolkowicz    proved that the following semidefinite programming (SDP)  relaxation is tight for   TTRS if and only if the Hessian of   Lagrangian is positive semidefinite at   global solution \cite{NesWolk:00}:
\begin{align}\label{relax}
\min_{x,X}&~ \frac12 \textbf{•}\trace(AX)+a^Tx \nonumber \\
&~  \trace(X)\leq \delta_1^2, \nonumber \\
&~\trace(BX)-2c^TBx+c^TBc \leq \delta_2^2   \tag{\rm {SDP}},\\ %
 &  X\succeq xx^T.  \nonumber
\end{align}
 In 2001, Chen and Yuan   presented a sufficient  condition
 under which the Lagrangian function  of the  CDT  subproblem has  positive semidefinite Hessian at optimal solution.
Moreover, Ye and Zhang \cite{YYe:15} showed that  for  general   CDT   subproblem  with certain additional conditions, the SDP relaxation is tight in many cases.  In 2005, Li and Yuan proposed an algorithm that finds a global solution of the  CDT  subproblem with no duality gap, i.e.,
the Hessian of   Lagrangian is positive semidefinite at   global solution.
Beck and   Eldar  \cite{BeckEldar:06} used the complex valued  SDP    approach to come up
with a similar sufficient condition to guarantee the positive semidefiniteness of the Hessian
of the Lagrangian function at optimal solution. They reported that in their experiments
on randomly generated instances, their sufficient condition was satisfied for the
majority of  instances.
In 2009, Ai and Zhang \cite{WAiZhang:15} derived   verifiable conditions to characterize when the   CDT   subproblem has no duality gap, which is equivalent to when the  SDP  relaxation of the  CDT  problem is tight.
In 2013, Burer and Anstreicher \cite{Buran:13} provided a tighter relaxation  by adding second order cone constraints to the classical  SDP  relaxation, but the resulting problem still has a relaxation gap.
Later  in 2016, Yang and Burer \cite{YanBur:16} reformulated  special case of the   TTRS
with two variables into an exact  SDP  formulation by adding valid constraints.
In  general,  the complexity of the  CDT   subproblem had been open for a long time,
until Bienstock  \cite{Bienstock:16} recently proved its polynomial-time solvability.
 Unfortunately, Bienstock's polynomial-time algorithm does not appear to be very practical, because the polynomial-time feasibility algorithm looks difficult to implement.   In the most recent research, Sakaue et. al \cite{SakaueNakat:16} proposed a polynomial-time algorithm assuming exact eigenvalue computation. However, due to the high computational cost of their algorithm, they reported numerical results for only  dimension $n\leq 40$.

%%%%%%%%%%%%%%%%%%%%%%%%%%%%%%%%%%%%%%%%
%%%%%%%%%%%%%%%%%%%%%%%%%%%%%%%%%%%%%%%%%%%
In several recent research,  {\it Alternating Direction Method of Multipliers} (ADMM) has been successfully used
to solve both  convex and nonconvex optimization problems with convergence analysis to stationary solutions \cite{Bai:15,  Boydadmm:10, haji, Hong:15, Luo:07, Y.Shen:14, L.Xu:11, SalahiTaati:17}. Moreover, global and local non-global minimizers of TRS are potential candidates for the optimal solution of \TTRS in the case of feasibility. Thus in this paper, we propose a hybrid of efficient algorithms for finding the global and local non-global minimizers of TRS and ADMM to solve \TTRSp \cite{Adachi:15,SalahiTaatiw:20}. The rest of the paper is organized as follows. In Section 2, we review some results related to \LNGM of \TRS and      optimality conditions  for \TTRSp. In  Section 3, we describe the hybrid algorithm and prove the convergence of ADMM steps to the first-order stationary point. Finally, we report   numerical results for several classes of test problems in Section 4 to demonstrate the efficiency of  hybrid algorithm
compared with the algorithm of Sakaue et. al \cite{SakaueNakat:16} for small dimensions and Snopt for medium and large-scale problems.\\
\\

{\hspace{-.5cm}\bf{Notations:}} The $i$th eigenvalue of $A$ is denoted by $\lambda_i$, where
 \begin{align*}
 \lambda_{\min}(A)=\lambda_1 \leq \lambda_2 \leq  \cdots \leq \lambda_n.
 \end{align*}
 Also  $A=Q\Lambda Q^T$ is the spectral decomposition of $A$ where
$\Lambda =\mathrm{Diag} (\lambda_1, \lambda_2, \cdots, \lambda_n)$ and  $q_i$ denotes the $i$th column of $Q$. The orthogonal complement of $W$ is
  \begin{align*}
W^{\bot } = \lbrace x~|~x^Ty=0,~~\forall y \in W \rbrace
 \end{align*}
and $\mathcal{N}(A)$ denotes the nullspace of $A$.

\section{{\bf \LNGM of \TRS and optimality conditions}}
\subsection{{\bf \LNGM of \TRS}}
In this subsection,   we review some results related to \TRSp.
Consider the following \TRS
by removing the second constraint of \TTRSp:
 \begin{align}\label{trsAD}
 &\min ~~ \frac{1}{2} x^TAx+ a^Tx \nonumber \\
 &~\mathrm{s.t.} ~~~~~ ||x||^2 \leq \delta_1^2.
 \end{align}
 If the global solution  of  (\ref{trsAD}) is  feasible for \TTRSp, then it is also a global solution for it. Otherwise,  (\ref{trsAD})  might have a \LNGM \ that  is feasible for \TTRS and also a candidate for its global solution. In what follows, we review some results related to LNGM of (\ref{trsAD}).
\begin{theorem}[Necessary Conditions for LNGM, \cite{Mart:94}]\label{nclng1}
 Let $x^*$ be a LNGM of (\ref{trsAD}). Choose $V \in \mathbb{R}^{n \times (n-1)}$ such that $\left[\frac{1}{||x^*||}x^* \Big{\vert} V \right]$
 is orthogonal. Then there exists a unique $ \lambda^* \in \left( \max \lbrace 0 ,-\lambda_2 \rbrace , -\lambda_1  \right)$   such that
 \begin{align}\label{theolng3}
& V^T\left( A+\lambda^* I \right) V\succeq 0, \nonumber \\
 &\left( A+\lambda^* I \right) x^* = -a, \\
& ||x^*||^2=\delta^2_1. \nonumber
 \end{align}
\end{theorem}
\begin{corollary}[Lemma 3.2, \cite{Mart:94}]\label{colng1}
 If  $a$ is orthogonal to
some eigenvectors corresponding to $\lambda_1$, then no LNGM exists.
\end{corollary}
Now let
$$ \phi (\lambda) := || \left( A+\lambda I \right)^{-1}  a||^2. $$
For
$$\lambda \in \left( \max \lbrace 0 ,-\lambda_2 \rbrace , -\lambda_1  \right),$$
Theorem \ref{nclng1} shows that  equation $\phi (\lambda)=\delta_1^2$
is a necessary condition for an LNGM.
Furthermore, using the eigenvalue decomposition of $A$, we have
\begin{align}\label{flng6}
&\phi (\lambda)= \sum_{i=1}^n \frac{(q^T_ia)^2}{(\lambda_i + \lambda)^2}, \nonumber \\
&\phi' (\lambda)= -2\sum_{i=1}^n \frac{(q^T_ia)^2}{(\lambda_i + \lambda)^3}, \nonumber \\
&\phi'' (\lambda)= 6\sum_{i=1}^n \frac{(q^T_ia)^2}{(\lambda_i + \lambda)^4}.
\end{align}
Equation (\ref{flng6}) implies that the function $\phi (\lambda)$ is strictly convex on $\lambda \in \left( \max \lbrace 0 ,-\lambda_2 \rbrace , -\lambda_1  \right)$
and so it has at most two roots in the interval $\left( \max \lbrace 0 ,-\lambda_2 \rbrace , -\lambda_1  \right)$, which leads to the following theorem.
\begin{theorem}[Theorem 3.1, \cite{Mart:94}] \label{lng1}
$ $ \\
 1. If $x^*$ is a LNGM of (\ref{trsAD}), then (\ref{theolng3}) holds with a unique $\lambda \in \left( \max \lbrace 0 ,-\lambda_2 \rbrace , -\lambda_1  \right)$ and  $\phi' (\lambda^*)\geq 0$. \\
 2. There exists at most one LNGM.
\end{theorem}
Based on this intuition,
given an instance of \TRS with global minimizer $x^*$ and LNGM
$\bar{x}$, one can enforce another ellipsoid that cuts off $x^*$ but leaves $\bar{x}$ feasible. For the resulting instance of \TTRSp, $\bar{x}$ becomes a natural candidate for the  optimal solution of \TTRSp, although points near $x^*$ that remain feasible are good candidates as well.
 In this paper, we take advantage of the efficient algorithm developed in \cite{SalahiTaatiw:20} to find LNGM within the proposed   hybrid algorithm  in the next section.

It is worth noting that the \TRS by removing the first constraint of \TTRS  is as following:
\begin{align}\label{trsAB}
\min \quad &\frac12x^TAx +a^Tx \notag \\
&\hspace{-.7cm}{\rm s.t.}\ (x-c)^TB(x-c) \leq \delta_2^2
\end{align}
which can be easily   transformed  to  (\ref{trsAD})  by   change of variables.
 %%%%%%%%%%%%%%%%%%%%%%%%
\subsection{{\bf Optimality conditions and strong duality for \TTRS}}
Let $x \in \mathbb{R}^n$ be a local solution
of \TTRS that satisfies the linear independence constraint qualification (LICQ). Then there exists a pair of Lagrange multipliers
$(\gamma, \mu) \in \mathbb{R}^2$ satisfying the KKT conditions:
\begin{align}\label{KKTTT1}
 &H(\gamma, \mu)x=-a  + \mu B c, \nonumber \\
  &||x||^2\leq \delta_1^2,~~~(x-c)^TB(x-c) \leq \delta_2^2,  \nonumber \\
   &\gamma\left( ||x||^2- \delta_1^2 \right)=0, \\
  &\mu\left( (x-c)^TB(x-c) - \delta_2^2 \right)=0, \nonumber  \\
&\gamma\geq 0,~~\mu \geq 0,  \nonumber
\end{align}
 where $ H(\gamma, \mu) = A+\mu B + \gamma I_n $
is the Hessian of the Lagrangian.
\begin{theorem}[\cite{Bomze:15}]\label{bomse91}
Let at  the KKT point $\bar{x}$, both constraints of \TTRS be active. \\
(a) If $H(\bar{\gamma}, \bar{\mu})$ is  positive definite   in $\lbrace \bar{x}, B(\bar{x}-c) \rbrace^\perp$ for some multipliers pair $(\bar{\gamma}, \bar{\mu}) \in \mathbb{R}^2_+$  satisfying
the KKT conditions (\ref{KKTTT1}), then $\bar{x}$ is a local minimizer of \TTRSp. \\
(b) If $\bar{x}$ is a local minimizer of \TTRSp, then $H(\bar{\gamma}, \bar{\mu})$ is  positive semidefinite in $\lbrace \bar{x}, B(\bar{x}-c) \rbrace^\perp$ for some multipliers
pair $(\bar{\gamma}, \bar{\mu}) \in \mathbb{R}^2_+$  satisfying the KKT conditions (\ref{KKTTT1}).
\end{theorem}
The following theorem gives necessary optimality conditions for \TTRSp.
\begin{theorem}[Theorem 4.3, \cite{Peng:97}]\label{negeig}
Let $x^*$ be a global minimizer of   \TTRSp. If  $x^*$ satisfies LICQ, then there exist $\gamma^*,\mu^* \in \mathbb{R}^+$  such that
$$ H(\gamma^*, \mu^*) x^* = \mu^* Bc -a,$$
and $ H(\gamma^*, \mu^*)  $ has at least $n-1$ nonnegative eigenvalues.
\end{theorem}
A sufficient condition for the optimality of minimizing
a quadratic function over two quadratic inequality constraints, when one of them is
strictly convex, is presented in  Theorem 2.1 of \cite{WAiZhang:15}. Since \TTRS is a special case,  we
immediately have the following theorem.
\begin{theorem}[Theorem 2.5, \cite{Yuan:90}] \label{globalsemi}
Let $x^*$ be feasible for \TTRSp. If there are two multipliers $\gamma^*,\mu^* \in \mathbb{R}^+$
 such that
 \begin{align*}
   &H(\gamma^*, \mu^*) x^* =  \mu^* Bc -a, \\
   &\gamma^* \left( ||x^*||^2 -\delta_1^2 \right) =0, \\
   &\mu^* \left(  (x^*-c)^T B (x^*-c) -\delta_2^2 \right) =0, \\
  &H(\gamma^*, \mu^*)  \succeq 0,
 \end{align*}
 then $x^*$  is a global solution of  \TTRSp.
\end{theorem}
\begin{theorem}
Suppose that strong duality holds for \ref{mrs}.  If  its optimal solution  $x^*$ is not the optimal solution of \TRS   (\ref{trsAD}) and (\ref{trsAB}), then
 \begin{align*}
  &||x^*||^2=\delta_1^2, \\
  &(x^*-c)^TB(x^*-c)=\delta_2^2.
 \end{align*}
\end{theorem}
\begin{proof}
Let  $(x^*-c)^TB(x^*-c)<\delta_2^2$, then either $x^*$ is the global minimizer or a LNGM of  \TRS  (\ref{trsAD}). Since strong
duality fails at LNGM, we conclude  it is the global minimizer of
\TRS  (\ref{trsAD}) which is in contradiction with our assumption. Thus
  $(x^*-c)^TB(x^*-c)=\delta_2^2$.
Similarly, it can be proved that  $||x^*||^2=\delta_1^2$.
\end{proof}
 We now discuss special cases of \TTRS where strong duality holds.
\begin{theorem}\label{theos1}
Suppose that $A=\mathrm{diag}(\delta_1, \cdots, \delta_n)$  and $B=\mathrm{diag}(\alpha_1, \cdots, \alpha_n)$ such that
for all $i<j$,  $\delta_i \leq \delta_j$ and  $\alpha_i \leq \alpha_j$. If   $\delta_1 = \delta_2$ and $\alpha_1 =\alpha_2$, then
 strong duality holds for \TTRSp.
\end{theorem}
\begin{proof}
Let $\bar{x}$ be a global optimal solution   of \TTRS that satisfies LICQ. Then   there exist two non-negative multipliers $\gamma$ and $\mu$ such that   KKT conditions (\ref{KKTTT1}) hold. We have
\begin{align} \label{fthes1}
 A+ \gamma I_n +\mu B &=\mathrm{diag}(\delta_1, \delta_2, \cdots, \delta_n) + \gamma \mathrm{diag}(1, 1, \cdots, 1)
 +\mu \mathrm{diag}(\alpha_1, \alpha_2, \cdots, \alpha_n) \nonumber \\
 & =\mathrm{diag}(\delta_1 + \gamma+ \mu \alpha_1, \delta_2 + \gamma+ \mu \alpha_2, \cdots, \delta_n + \gamma+ \mu \alpha_n).
\end{align}
Then,
 \begin{align*} %\label{fthes2}
  \forall i \in \lbrace 3 ,\cdots, n  \rbrace , ~~\delta_1 + \gamma+ \mu \alpha_1= \delta_2 + \gamma+ \mu \alpha_2 \leq \delta_i + \gamma+ \mu \alpha_i.
\end{align*}
If $\delta_1 + \gamma+ \mu \alpha_1<0$, then $H(\gamma, \mu)$ has two   negative eigenvalues  which   contradicts  Theorem \ref{negeig}.\qed
\end{proof}
\begin{theorem}
Suppose that $A=\mathrm{diag}(\delta_1, \cdots, \delta_n)$  and $B=\mathrm{diag}(\alpha_1, \cdots, \alpha_n)$ such that
for all $i<j$, $\delta_i \leq \delta_j$ and  $\alpha_i \leq \alpha_j$. If   $\delta_1 = \delta_2=\delta_3$ and $\alpha_1 =\alpha_2=\alpha_3$, then any local minimum of \TTRS is  a global minimum.
\end{theorem}
\begin{proof}
Let $\bar{x}$ be a local minimum   of \TTRS that satisfies LICQ. Then  there  exist two nonnegative multipliers $\gamma$ and $\mu$ such that  satisfy   KKT conditions (\ref{KKTTT1}). From $\delta_1 = \delta_2=\delta_3$, $\alpha_1 =\alpha_2=\alpha_3$ and (\ref{fthes1}), we have
 \begin{align} \label{fthes2}
 \forall i \in \lbrace 4 ,\cdots, n  \rbrace,~~\delta_1 + \gamma+ \mu \alpha_1 &= \delta_2 + \gamma+ \mu \alpha_2 \nonumber \\
&= \delta_3 + \gamma+ \mu \alpha_3 \leq \delta_i + \gamma+ \mu \alpha_i.
\end{align}
If $\delta_1 + \gamma+ \mu \alpha_1<0$, then  the  three smallest eigenvalues of $ H(\gamma, \mu)$ are  negative, which  contradicts the second order necessary optimality condition   for \TTRSp. Therefore, any local minimum of TTRS is  a global minimum.\qed
\end{proof}
%%%%%%%%%%%%%%%%%%%%%%%%%%%%%%%%%%%%%%%%%%%%%%%%%%%%%%%
\section{Hybrid  algorithm}
In this section, we  present the hybrid algorithm  for solving \TTRSp. Before starting the ADMM steps, the feasibility of \TTRS and the  feasibility of global solutions of (\ref{trsAD}) and (\ref{trsAB})  for \TTRS are checked. The feasibility of \TTRS is checked by solving the following
\TRSp:
 \begin{align}\label{trsD}
v_{ch}^*=\min \quad &(x-c)^TB(x-c) - \delta_2^2, \notag \\
&\hspace{-.7cm}{\rm s.t.}\ ||x||^2 \leq \delta_1^2.
\end{align}
%let $v^*$ be optimal value of  (\ref{trsD}), if $v^*<0$  then feasible region of (\ref{mrs})    is nonempty.
If $v_{ch}^*\leq 0$,  then \TTRS is feasible, otherwise it is infeasible.

In the case of  feasibility of \TTRSp, we check whether the optimal solution of \TRS (1) and (4) are optimal for \TTRSp. To do so, consider the following sets:
\begin{align*}
&E_1=\lbrace x~|~||x||^2 \leq \delta_1^2\rbrace, \\
&E_2=\lbrace x~|~(x-c)^TB(x-c) \leq \delta_2^2 \rbrace,\\
&E=E_1 \cap E_2.
\end{align*}
Let $x^*$ be the optimal solution of \TRS (\ref{trsAD}). If $x^* \in E$, then $x^*$  is a global
minimizer of \TTRSp. Also if the optimal solution of \TRS  (\ref{trsAB}) belongs to $E$, then it is optimal for \TTRSp.
%%%%%%%%%%%%%%%%%%%%%%%%%%%%%%%%%%%%%%%

It is worth mentioning that for (\ref{trsAD}) and (\ref{trsAB})  hard case  may occur.
Therefore, we should check whether these two problems have another optimal solution   in  the feasible region of  \TTRSp. This  procedure is  discussed for  (\ref{trsAB}) since    (\ref{trsAD}) is a special case of (\ref{trsAB}).
\begin{definition}[\cite{Adachi:15}]
A TRS   is   "hard case", if  $\mu^*=\lambda_n(A+\mu^*B) $, the largest generalized eigenvalue of the pencil $A+\mu B  $.
\end{definition}
\begin{theorem}[\cite{Adachi:15}]
A vector $x^*$ is an optimal solution to the \TRS (\ref{trsAB})   if and only if there exists $\mu^* \geq 0$ such that
\begin{align}
&(x^*-c)^TB(x^*-c) \leq  \delta_2^2,  \label{kktt1}\\
&(A+\mu^*B)x^* = \mu^* Bc -a,  \label{kktt2}\\
&\mu^* \left((x^*-c)^TB(x^*-c)- \delta_2^2  \right) =0,  \label{kktt3}\\
&A+\mu^*B \succeq 0. \label{kktt4}
\end{align}
\end{theorem}
\begin{theorem}[\cite{Adachi:15}]
 Suppose  \TRS  (\ref{trsAB}) belongs to the "hard case" and
 $(\mu^*; x^*)$ satisfies (\ref{kktt1})-(\ref{kktt4}) with $\mu^*=\lambda_n(A+\mu B) $.
 Let $r =\mathrm{dim}\left(\mathcal{N}(A+\mu^*B)\right) $ and $V := [v_1, \cdots, v_r] $ be a basis of  $\mathcal{N}(A+\mu^*B)$ that is $B$-orthogonal, i.e., $V^TBV = I_r$. For an arbitrary $\sigma > 0$, define
\begin{align*}
&H :=A+\mu^*B + \sigma \sum_{i=1}^r Bv_i v_i^TB.
\end{align*}
Then $H$ is positive definite. Moreover, $q := -H^{-1}\left( A c + a \right)+c$ is the minimum-norm solution
to the linear system $(A+\mu^*B)x = \left(  \mu^*B c - a \right)$ in the $B$-norm, that is,
\begin{align}
 q =\mathrm{argmin}  \lbrace ||x-c||_B~|~ (A+\mu^*B)x = \mu^* Bc - a  \rbrace.
\end{align}
Furthermore,   there exists      $\alpha \in \mathbb{R}^r$ such that
 $x^*=q+V \alpha$ is a solution for \TRS (4).
\end{theorem}
%Therefore,  \TRS  have  several   solution, if it is "hard case".
Now let $X=\lbrace x^*~|~ x^* = q + V\alpha,~||x^*||_B^2 =\delta_2^2 \rbrace$   be the set  of optimal solutions of (\ref{trsAB}), where  $\alpha \in \mathbb{R}^r$.
To see whether there exists an optimal solution of \TRS (4) that belongs to $E_1$, it is sufficient to solve the following problem:
\begin{align}\label{trsDB}
\min \quad &||q + V \alpha||^2 \notag \\
&\hspace{-.7cm}{\rm s.t.}\ (q + V \alpha-c)^TB(q + V \alpha-c) = \delta_2^2.
\end{align}
This  is equivalent to the  following problem:
\begin{align*}
\min \quad &\alpha^T  \left[ V^TV \right]\alpha + 2q^T V \alpha  \notag \\
&\hspace{-.7cm}{\rm s.t.}\ \alpha^T \alpha = \delta_2^2 - (q-c)^TB(q-c).
\end{align*}
Let $\alpha^*$ be its optimal solution, and set $x^* =q + V \alpha^*$. If $||x^*||^2 \leq \delta_1^2 $, then it is  an optimal solution of \TTRSp. The same procedure should be performed for \TRS (1).
%%%%%%%%%%%%%%%%%%%%%%%%%%%%%%%%%%

After the above discussion, if none of the optimal solutions of \TRS (1) and (4) are feasible for \TTRSp, we move to the steps of ADMM steps.
One can write TTRS in the following equivalent form:
\begin{align}\label{TTRSLR}
\min \quad &\frac12x^TAx +a^Tx \nonumber \\
&||z||^2 \leq \delta_1^2 ,\\
&(x-c)^TB(x-c) \leq \delta_2^2 , \nonumber \\
&x=z.\nonumber
\end{align}
To define the ADMM steps,
consider the following {\it augmented  Lagrangian} for (\ref{TTRSLR}):
\[L(x,z,\lambda)=\frac12x^TAx +a^Tx+\lambda^T(x-z)+\frac{\rho}{2} ||x-z||^2,\]
where $\lambda_i$'s are Lagrange multipliers and $\rho>0$ is the penalty parameter. The ADMM steps for the given $x^k$ and $\lambda^k$ are as follow
 \cite{Boydadmm:10}:
\begin{itemize}
\item {\bf Step 1:} $z^{k+1}={{\rm{argmin}}_{||z||^2 \leq \delta_1^2}} L(x^{k},z,\lambda^k).$
\item{\bf Step 2:}  $x^{k+1}={\rm{argmin}}_{(x-c)^TB(x-c) \leq \delta_2^2 } L(x,z^{k+1},\lambda^k).$
\item {\bf Step 3:} $\lambda^{k+1}=\lambda^k+\tau\rho(x^{k+1}-z^{k+1}),$ where $\tau \in (0,1)$ is a constant.
\end{itemize}
In Step 1, we   solve the following \TRSp:
 \begin{eqnarray}\label{step1}
\min &  L(x^{k},z,\lambda^k) \nonumber\\
        &\hspace{-0.6cm}{\rm s.t.}\  ~~ ||z||^2\leq \delta_1^2.
\end{eqnarray}
Let $z^{k+1}$ be the optimal solution of (\ref{step1}). In Step 2, we  solve the following \TRSp:
\begin{eqnarray}\label{step2}
\min &  L(x,z^{k+1},\lambda^k) \nonumber\\
        &\hspace{-0.6cm}{\rm s.t.}\  ~~~~~~(x-c)^TB(x-c)\leq \delta_2^2.
\end{eqnarray}
%Let   $x^{k+1}$ be an optimal solution of (\ref{step2}).
As we see, in each step, we  need to solve a  \TRS for which  efficient algorithms  are available  \cite{Adachi:15,con}.\\
\\
---------------------------------------------------------------------------------------------------\\
{\bf {Hybrid  algorithm}}\\
--------------------------------------------------------------------------------------------------- \\
{\bf Step 0-1:} Check the feasibility of  \TTRS by solving (\ref{trsD}). If $v_{ch}^* >0$ then  \TTRS is infeasible, { exit};
{else} go to Step 2.   \\
{\bf Step 0-2:} Solve both \TRS   (\ref{trsAD})  and (\ref{trsAB}).
   If $x_1^* \in E$  or $x_2^* \in E$, then exit with  the global solution of  \TTRSp;
   { else} go to Step 3.    \\
{\bf Step 0-3:} Compute the  LNGM of (\ref{trsAD}) and (\ref{trsAB})  if they exist. Keep them if they are feasible for  \TTRSp. \\
 {\bf ADMM steps:}\\
 {\bf Input parameters:} $tol>0$,  maxiter$>0$. Choose appropriate penalty parameter  $\rho>0$ and $\tau>0$. Set $k=0$ and choose appropriate
$x^k$ and $\lambda^k$\\
{\bf For} $k=1,\cdots,$ maxiter {\bf do}\\
Solve   \TRS    (\ref{step1}) and let $z^{k+1}$ be its optimal solution. \\
Solve   \TRS   (\ref{step2}) and let $x^{k+1}$ be its optimal solution.\\
 {\bf If} $||x^{k+1}-z^{k+1}||\le tol$, then exit with $x^{k+1}$ as output.\\
 {\bf end if}\\
 Set  $\lambda^{k+1}=\lambda^k+\tau\rho(x^{k+1}-z^{k+1})$ and $k=k+1.$\\
 {\bf end for}.\\
  Choose   $x^*$ as the best of ADMM steps and   LNGM of  (\ref{trsAB}) and (\ref{trsAD}) if they exist.  \\
 -----------------------------------------------------------------------------------------------------
 \\
In what follows, if the global minimum of  \TTRS is not the global or \LNGM of TRS (1) or (4),  we discuss the convergence of the  ADMM steps
to the stationary point of  \TTRSp. First we present the following lemma.
\begin{lemma}
Suppose that $\{\lambda^k\}$ is bounded and  $\sum_{k=1}^{\infty} ||\lambda^{k+1}-\lambda^k||^2<\infty.$ Then
\[||x^{k+1}-x^k||\rightarrow 0,\ \ ||z^{k+1}-z^k||\rightarrow 0\ \ {\rm as}\ \ k\rightarrow \infty .\]
\end{lemma}
\begin{proof} Since $x^{k+1}$ solves problem (\ref{step2}) at $k$-th
iteration and $x^k-x^{k+1}$ is a feasible direction with respect to
the feasible region of (\ref{step2}), then
\begin{eqnarray}\label{th1eq1}\nabla_xL(x^{k+1},z^{k+1},\lambda^k)^T(x^k-x^{k+1})\ge
0.\end{eqnarray} Moreover,
\begin{eqnarray}&&L(x^k,z^{k+1},\lambda^k)-L(x^{k+1},z^{k+1},\lambda^k)= \frac12(x^k-x^{k+1})^T(A+\rho I)(x^k-x^{k+1})\nonumber\\
&&+\nabla_xL(x^{k+1},z^{k+1},\lambda^k)^T(x^k-x^{k+1})\ge
\frac{\lambda_{1}+\rho}{2}||x^k-x^{k+1}||^2,\end{eqnarray}
where the
inequality follows from the definition of the smallest eigenvalue of
$A$, $\lambda_1$,  and (\ref{th1eq1}). We also have
\begin{eqnarray}\label{th1eq2}L(x^k,z^{k},\lambda^k)-L(x^{k},z^{k+1},\lambda^k)\ge
0,\end{eqnarray} as $z^{k+1}$ is the minimizer of
$L(x^k,z,\lambda^k)$. On the other hand
\begin{eqnarray}\label{th1eq3}L(x^{k+1},z^{k+1},\lambda^k)-L(x^{k+1},z^{k+1},\lambda^{k+1})&& = (\lambda^k-\lambda^{k+1})^T(x^{k+1}-z^{k+1})\nonumber\\ && =-\frac{1}{\tau\rho}||\lambda^k-\lambda^{k+1}||^2.\end{eqnarray}
Now using (\ref{th1eq1}), (\ref{th1eq2}), and (\ref{th1eq3}) we have
\begin{eqnarray}\label{lag1}L(x^k,z^{k},\lambda^k)&&-L(x^{k+1},z^{k+1},\lambda^{k+1})=L(x^k,z^{k},\lambda^k)-L(x^{k},z^{k+1},\lambda^{k})\nonumber\\
&& \hspace{-2.5cm}+L(x^k,z^{k+1},\lambda^k)-L(x^{k+1},z^{k+1},\lambda^{k})+L(x^{k+1},z^{k+1},\lambda^k)-L(x^{k+1},z^{k+1},\lambda^{k+1})\nonumber\\
&& \ge
\frac{\lambda_1+\rho}{2}||x^{k+1}-x^k||^2-\frac{1}{\tau\rho}||\lambda^{k+1}-\lambda^k||^2.\end{eqnarray}
Since $\{\lambda^k\}$ and $\{x^k\}$ are bounded,  from Step 3 of
ADMM iterations,  $\{z^k\}$ is also  bounded. Thus
$\{L(x^k,z^k,\lambda^k)\}$ is bounded. Moreover, since by assumption
$\sum_{k=1}^{\infty}||\lambda^{k+1}-\lambda^k||^2<\infty,$ then from
(\ref{lag1}), $\sum_{k=1}^{\infty}||x^{k+1}-x^k||^2$ is a bounded
series (in the sense that the sequence of partial sums is
bounded) with nonnegative terms, thus it is convergent.
Therefore $||x^k-x^{k+1}||\rightarrow 0,\ \ {\rm as}\ \ k\rightarrow
\infty.$ Moreover since by assumption
$||\lambda^k-\lambda^{k+1}||\rightarrow 0,\ \ {\rm as}\ k\rightarrow
\infty,$ from the Step 3 we have $x^k-z^k\rightarrow 0,\ \ {\rm as}\
k\rightarrow \infty.$ Finally since
\[z^k-z^{k+1}=z^k-x^k+x^k-x^{k+1}+x^{k+1}-z^{k+1},\] and
we know $z^k-x^k\rightarrow 0,\  x^k-x^{k+1}\rightarrow 0,\
x^{k+1}-z^{k+1}\rightarrow 0,\ \ {\rm as}\ k\rightarrow\infty,$ then
$z^k-z^{k+1}\rightarrow 0$.\qed
\end{proof}
In what follows, the convergence of algorithm to the first-order stationary conditions is proved.

\begin{theorem}
Let $(x^*,z^*,\lambda^*)$ be any accumulation point of $\{(x^k,z^k,\lambda^k)\}$ generated by the  ADMM steps. Then by
boundedness of $\{\lambda^k\}$ and $\sum_{k=1}^{\infty} ||\lambda^{k+1}-\lambda^k||^2<\infty,$ $x^*$ satisfies the first-order stationary conditions.
\end{theorem}
\begin{proof}
Since $(x^*,z^*,\lambda^*)$ is an accumulation point of $\{(x^k,z^k,\lambda^k)\}$, then there exists a subsequence $\{(x^k,z^k,\lambda^k)\}_{k\in I}$
that converges to $(x^*,z^*,\lambda^*)$. Now consider subproblems  that should be solved in Steps 1 and 2 of ADMM iterations.
 Subproblem (\ref{step1}) in Step 1 is  a convex quadratic optimization problem which its necessary and sufficient optimality conditions are  as follow:
\begin{eqnarray}\label{trs2}
&& \rho z^{k+1}-(\lambda^k+\rho x^{k})+  \gamma^{k+1} z^{k+1} = 0,\nonumber\\
&&  ||z^{k+1}||^2 \leq \delta_1^2, \\
&& \gamma^{k+1} \left[ ( ||z^{k+1}||^2 - \delta_1^2 \right]=0, ~~\gamma^{k+1} \geq 0, \nonumber\\
&& \rho I_n +  \gamma^{k+1} I_n\succeq 0, \nonumber
\end{eqnarray}
where $\gamma^{k+1}$  is the Lagrange multiplier.  Moreover, subproblem in Step 2 is a \TRS with  the following necessary and sufficient optimality conditions:
\begin{eqnarray}\label{trs1}
&& (A+\rho I_n )x^{k+1}+ a+\lambda^k - \rho z^{k+1}+ \mu^{k+1} B (x^{k+1} - c) = 0,\nonumber\\
&&  ( x^{k+1}- c)^T B ( x^{k+1}- c) \leq \delta_2^2, \nonumber \\
&& \mu^{k+1}\left[ ( x^{k+1}- c)^T B ( x^{k+1}- c)-  \delta_2^2 \right] =0,~~\mu^{k+1} \geq 0,\\
&& A+\rho I_n+\mu^{k+1}B\succeq 0_{n\times n}.\nonumber
\end{eqnarray}
Now by taking the limit of both (\ref{trs2}) and (\ref{trs1}), we
get
\begin{eqnarray}\label{trs11}
&& \rho z^{*}-(\lambda^*+\rho x^{*})+   \gamma^{*} z^{*} = 0, \label{KKT25}\\% \nonumber\\
&&  ||z^*||^2 \leq \delta_1^2,\label{KKT26} \\%\nonumber \\
&& \gamma^{*} \left[ ||z^*||^2 -  \delta_1^2 \right]=0, ~~\gamma^{*} \geq 0, \label{KKT27}\\% \nonumber\\
%&& \rho I_n +  2 \gamma^{*} I_n\succeq 0, \nonumber\\
&& (A+\rho I_n )x^{*}+  a+\lambda^* - \rho z^{*}+   \mu^{*} B (x^{*} - c) = 0,\label{KKT28}\\
&&  ( x^{*}- c)^T B ( x^{*}- c) \leq \delta_2^2,\label{KKT29} \\%\nonumber \\
&& \mu^{*}\left[ ( x^{*}- c)^T B ( x^{*}- c)-  \delta_2^2 \right] =0,~~\mu^{*} \geq 0.\label{KKT30}% \nonumber
%&& A+\rho I_n+2\mu^{*}B\succeq 0_{n\times n}.\nonumber
\end{eqnarray}
From  (\ref{KKT25}) and  (\ref{KKT29}), we get
$$Ax^{*}+ a +   \gamma^{*} x^{*} +    \mu^{*} B (x^{*} - c)  = 0,$$
which  with  (\ref{KKT26}), (\ref{KKT27})   and  (\ref{KKT29}), (\ref{KKT30})   are the first-order stationary conditions.\qed
\end{proof}

\section{Numerical experiments}
In this section, we present several classes of test problems to assess the performance of Hybrid algorithm for solving \TTRSp.
For small dimension problems, we compare { Hybrid} algorithm with the  SDP
relaxation  of \TTRS and     by Sakaue et. al's algorithm \cite{SakaueNakat:16}.
For large-scale problems, we do comparison with Snopt   through Tomlab as the software giving   best results.
All computations are performed  in MATLAB R2015a on a 2.50 GHz laptop with 8 GB of RAM. To solve the SDP reformulation, we have used CVX  1.2.1. For all test  problems, we set $tol=10^{-7}$ and ${{\rm maxiter}}=1000$.   To solve the  TRSs  within the algorithm we have used the algorithm in \cite{Adachi:15} and to find the \LNGM of \TRS we have used the algorithm of \cite{SalahiTaatiw:20}.
Finally, we should note that results in tables are the average of 100 runs for each dimension.

 Our  numerical experiments led to the following starting point procedure for  Hybrid algorithm.
 By considering $\beta_1, \beta_2 \in \mathbb{R}_+$, the parametric form of the \TTRS   is constructed as follows:
\begin{align}\label{mrs123}
\min \quad \frac12x^TAx +a^Tx &+\beta_1 \left( (x-c)^TB(x-c) - \delta_2^2 \right)  \notag\\
&||x||^2 \leq \delta_1^2,
\end{align}
\begin{align}\label{mrs124}
\min \quad &\frac12x^TAx +a^Tx +\beta_2 \left( ||x||^2 - \delta_1^2 \right)  \notag \\
&(x-c)^TB(x-c) \leq \delta_2^2.
\end{align}
Let $x^{*}_{\beta_1}$ and $x^{*}_{\beta_2}$ be optimal solutions of  (\ref{mrs123}) and (\ref{mrs124}), respectively.
For large enough $\bar{\beta}_1$ and $\bar{\beta}_2$, these solutions  are feasible for \TTRSp.
Now, consider the following starting point for   {Hybrid} algorithm:
\begin{align}\label{intialpoint}
x_0 = \operatornamewithlimits{\textsf{argmin}}_{x \in  \lbrace x^{*}_{\bar{\beta}_1},x^{*}_{\bar{\beta}_2}\rbrace}
 \lbrace \frac12  x^TAx +a^Tx\rbrace.
\end{align}
\begin{table}
\centering
\small
\begin{tabular}{c|cccccccc}
\hline
Notation&    Description        \\
\hline
n            &    Dimension of problem   \\
%\hline
Den       &Density of $A$ and $B$    \\
%\hline
CPU        &Run time         \\
%\hline
KKT         &   $||(A+\gamma^*I+ \mu^*B)x^*+ a - \mu^*Bc|| $  \\
%\hline
L(\ref{trsAD})            &  Number of times that LNGM of problem  (\ref{trsAD}) is feasible for TTRS    \\
%\hline
 L(\ref{trsAB})           & Number of times that LNGM of problem (\ref{trsAB})  is feasible for TTRS    \\
%\hline
Opt-2active        &  Number of times that both constraints of TTRS are active at optimality   \\
%\hline
    Opt-L (\ref{trsAD})      & Number of times that  LNGM  of  \TRS  (\ref{trsAD})  is optimal\\
%\hline
Opt-L  (\ref{trsAB})      &Number of times that  LNGM  of   \TRS (\ref{trsAB})   is optimal  \\
%\hline
Obj         &$\frac{1}{2}(x^*)^TAx^*+  ax^*$    \\
%\hline
{$F_{Hy}$} &Objective value of Hybrid algorithm    \\
%\hline
$F_{Sn}$ &Objective value of  Snopt solver of Tomlab   \\
%\hline
Cv-Snopt   &Number of times  the  feasibility of constraints is violated by Snopt   \\
%\hline
\hline
\end{tabular}
\caption{ \scriptsize  Notations in the tables}
\label{notationR}
\end{table}
\begin{itemize}
\item {\bf  First class of test problems:}\\
Here we consider two small dimensional  examples.\\
 {\bf Example 1:} Consider the following problem  which is taken   from \cite{Buran:13}:
\begin{align*}
\min ~~& x^T \begin{bmatrix}
  -4  ~~  &  1\\
  1   ~~  & -2
\end{bmatrix}
x+\begin{bmatrix}
   1\\
  1
\end{bmatrix}^Tx \\
s.t.~~~&||x||_2^2\leq 1,~~
x^T \begin{bmatrix}
  3 ~~  &  0\\
  0   ~~  & 1
\end{bmatrix}
x\leq 2.
\end{align*}
The optimal solutions are $x^* = (\pm 1 , \mp 1)/ \sqrt{2} $ with the  objective value $-4$.
The SDP relaxation   gives $-4.25$ and
by applying the approach of \cite{BurerYang:15}, one gets objective value $−4.0360$.
Our algorithm gives  $( \pm 0.70711, \mp 0.70711)$  with
objective value $-4.0000$, from any starting point. Both LNGM and global solutions of TRS  (\ref{trsAD})  and (\ref{trsAB}) are
 infeasible for TTRS.
%It is worth to note that  the problem does not have local solution.
\\
\\
 {\bf Example 2:}
 Consider the following problem where the two ellipsoids intersect at four points as shown   in Figure \ref{figexa2}:
  \begin{align*}
\min ~~& x^T \begin{bmatrix}
  -4  ~~  &  1\\
  1   ~~  & -2
\end{bmatrix}
x+\begin{bmatrix}
   1\\
  1
\end{bmatrix}^Tx \\
s.t.~~~&||x||_2^2\leq 1,~~
x^T \begin{bmatrix}
  \frac{9}{4} ~~  &  0\\
  0   ~~  & \frac{1}{4}
\end{bmatrix}
x\leq 1.
\end{align*}
Both \TRS \ (1) and (4) have no \LNGM and their global optimal solutions are not feasible for \TTRSp. Thus the optimal solution of \TTRS  is  on the intersection of two constraints i.e.,  one of the four points. The global solution   is $x^* = (\sqrt{3}, -\sqrt{5})/\sqrt{8} $ with
objective value $-3.8964$ and  two local solutions   are  $ (-\sqrt{3}, \pm\sqrt{5})/\sqrt{8}$  %(Theorem \ref{bomse91})
where the Hessian of Lagrangian has one negative eigenvalue. Also, $ (\sqrt{3}, \sqrt{5})/\sqrt{8}$  is the third local solution where Hessian of Lagrangian has two negative eigenvalues. Hybrid algorithm can   converge to any of these points,  depending  on the starting point. However, if we consider $(0.4054,  -0.9141)$ as starting point  from (\ref{intialpoint}),  Hybrid algorithm converges to $x^*$.
  % Initial point for  {\color{red} hybrid} algorithm uses (27)
\begin{figure}[h]
\centering
\includegraphics[width=10cm]{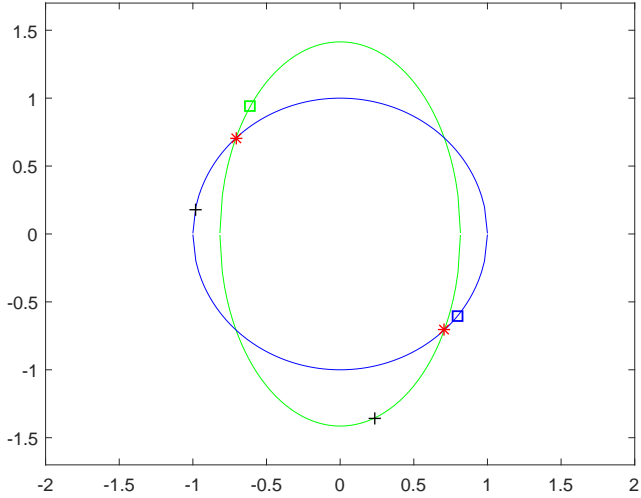}
\caption{\label{figexa1}\textsl{ \scriptsize
   $ +$: global solution of  (\ref{trsAD})  and   (\ref{trsAB}), $\Box$: local non-global solution of (\ref{trsAD})  and   (\ref{trsAB}).
   $\ast$: global solution of \TTRSp. The figure for Example 1 of  First class.
}}
\end{figure}
\begin{figure}[h]
\centering
\includegraphics[width=10cm]{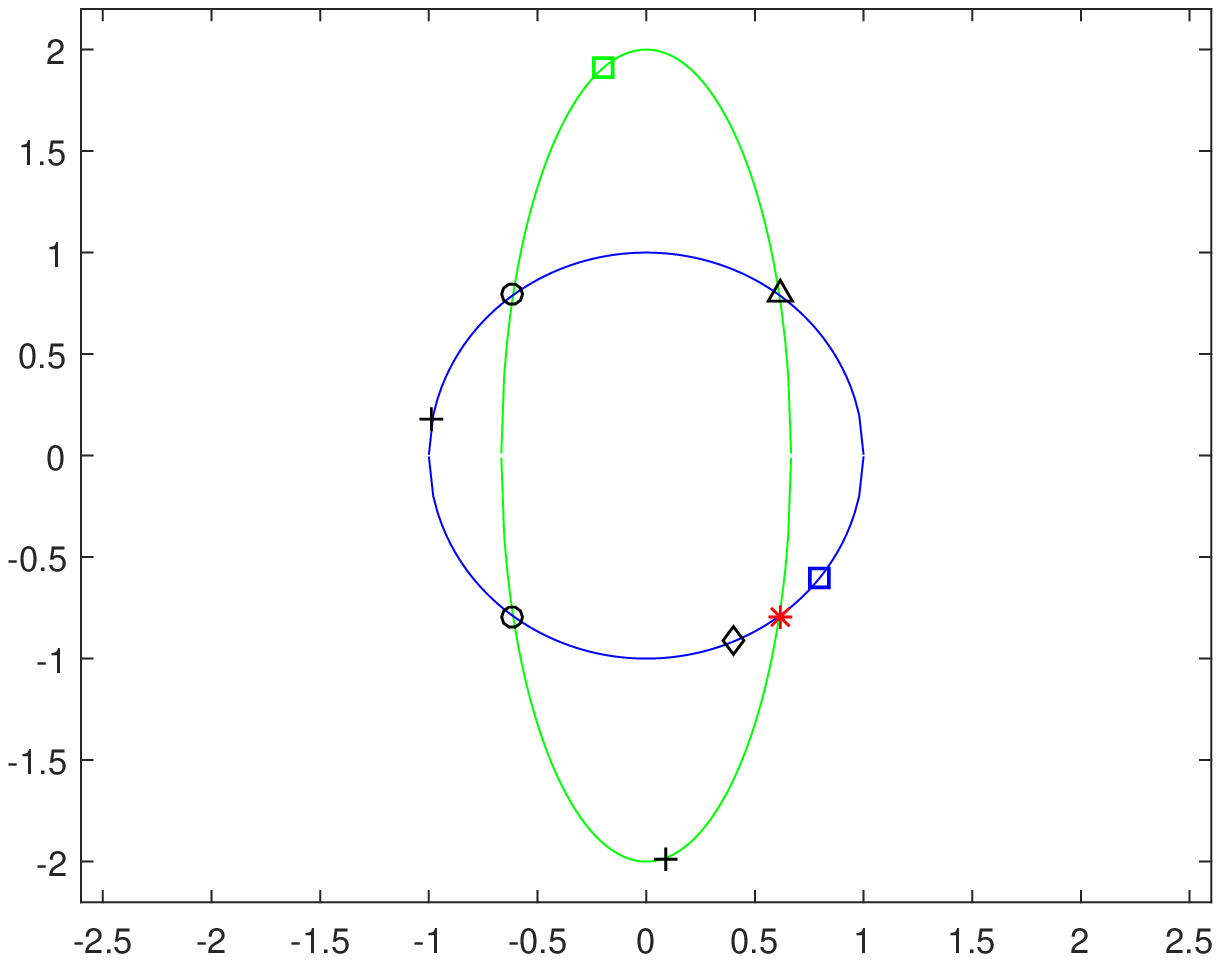}
\caption{\label{figexa2}\textsl{ \scriptsize
    $ +$: global solution of (\ref{trsAD})  and   (\ref{trsAB}), $\Box$: local non-global solution of (\ref{trsAD})  and   (\ref{trsAB}).
   $\ast$:  global  solution of \TTRSp.  $\bigcirc$: local  solution with one negative eigenvalue at Hessian of Lagrangian.
     $\bigtriangleup$: local  solution with two negative eigenvalue at Hessian of Lagrangian.
    $\Diamond$ : Starting point for  {Hybrid} algorithm (given by  (\ref{intialpoint})).
}}
\end{figure}
\\
%%%%%%%%%%%%%%%%%%%%%%%%%%%%%%%%%%%%%%%%%
%%%%%%%%%%%%%%%%%%%%%%%%%%%%%%%%%%%%%%%%%
\item {\bf Second class of test problems}\\
This class of test problems is generated using the following lemma \cite{SalahiTaatiw:20}.
\begin{lemma}\label{lemLNG}
Let  $A \in S^n$ and suppose that $\lambda_1 < \min \lbrace  0, \lambda_2  \rbrace$. Then there exists a linear term $a$ for which the eigenvector $v_1$ associated with $\lambda_1$ is the \LNGM of (\ref{trsAD}). Moreover $-v_1$ is the global  solution of  (\ref{trsAD}).
\end{lemma}
\begin{proof}
The first part of the proof is from \cite{SalahiTaatiw:20}.  Let $\mu_0 \in (\max \lbrace  0,- \lambda_2  \rbrace , -\lambda_1)$. Set $a=-\left( A+ \mu_0 I_n \right) v_1$ where $v_1$ is the eigenvector for  $\lambda_1$  with $||v_1||^2=\delta_1^2$.
For $\mu^* = -2\lambda_1-\mu_0$ and $x^*=-v_1$, the first order stationary condition of (\ref{trsAD}) holds:
\begin{align*}
\left( A + \mu^* I_n \right)  x^*+a &= \left( A + (-2\lambda_1-\mu_0) I_n \right)  (-v_1)- \left( A+ \mu_0 I_n \right) v_1 \\
& =- \left( A + (-2\lambda_1-\mu_0) I_n+A+\mu_0 I_n  \right)  v_1 \\
&=-2 \left( A -\lambda_1 I_n \right)  v_1=0.
\end{align*}
Moreover, since  $\mu^* = -2\lambda_1-\mu_0 > -\lambda_1$ then $A+\mu^* I \succ 0$, which implies
  that $-v_1$   is the global  minimizer  of  (\ref{trsAD}).
\end{proof}
To generate the desirable random instances
of \TTRSp, we proceed as follows. First we construct a \TRS
instance of the form  (\ref{trsAD})  having   LNGM
based on Lemma \ref{lemLNG}. Then we add the inequality constraint
$(x-c)^TB(x-c) \leq \delta_2^2$  to enforce that the global minimizer $-v_1$ of  \TRS be infeasible but  the \LNGM, $v_1$,  remains feasible  (Figure \ref{fig1}) for \TTRSp.
For $20\%$ of the generated instances, the LNGM of the  TRS   (\ref{trsAB})   is also in the feasible region of \TTRSp. Moreover,
strong duality fails at  $90\%$  of the generated instances.

For this class, we compare the {Hybrid} algorithm with Sakaue et. al's algorithm  \cite{SakaueNakat:16} and the  Snopt solver in Tomlab.  Our extensive testing showed that $\tau = 0.9$, $\rho = 4|\lambda_1(A)| + 1$, and $\lambda=2 x_0 $ are appropriate choices where $x_0$ is given by (\ref{intialpoint}).
 Results are summarized in Tables \ref{tabjacas2} to \ref{tabTo2cas2}   for the average of 100 runs.
In Table \ref{tabjacas2}, we compare the {Hybrid} algorithm with the Sakaue et. al's algorithm  \cite{SakaueNakat:16} for  dimension $n\le 30$.
It can be seen that, {Hybrid} algorithm is much faster than Sakaue et. al's algorithm  while having equal objective values (the difference is of $O(10^{-7})$) and comparable KKT accuracy\cite{SakaueNakat:16}. In Tables \ref{tabTo1cas2} and \ref{tabTo2cas2}, we compare {Hybrid} algorithm with the  Snopt solver of Tomlab for different densities. From these two tables,  we can conclude that {Hybrid} algorithm is much better than   Snopt   in large-scale problems.
%In Table \ref{tab1cas2}, we  report the properties of   this class of test problems.
%In this table, all of the instances have at least one \LNGM,  and the optimal value of  more than $90\%$  of them is LNGM.
In Table \ref{tab1cas2}, we have generated examples where the LNGM of TRS (\ref{trsAB}) is always feasible and for about $20\%$ of the generated instances, the LNGMs of \TRS (\ref{trsAD}) are also feasible. Moreover,  for over $95\%$ of instances, the optimal solution is at one of \LNGM, mostly on \LNGM \ of  TRS (\ref{trsAB}).
\begin{figure}[h]
\centering
\includegraphics[width=10cm]{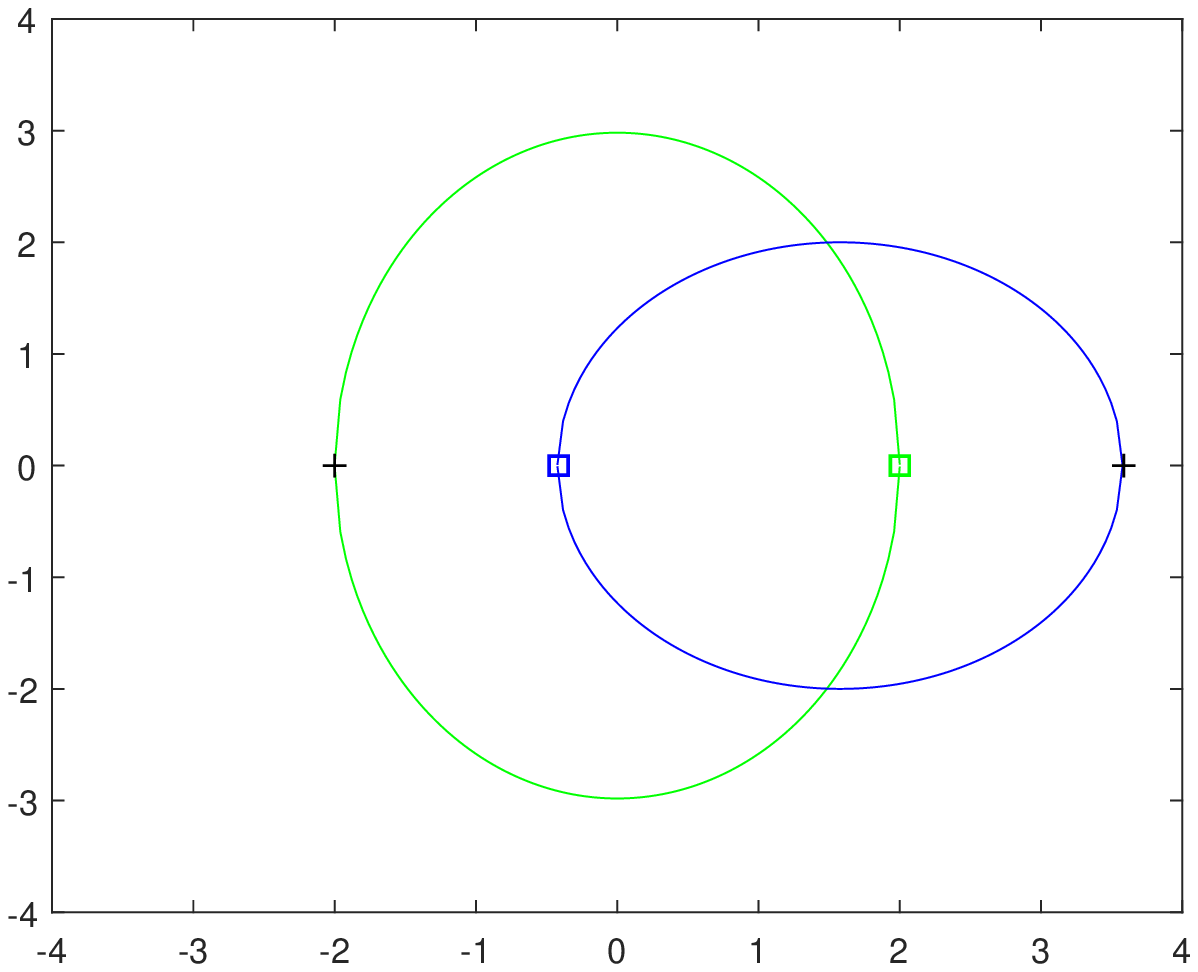}
\caption{\label{fig1}\textsl{ \scriptsize
   $ +$: global solution of (\ref{trsAB})  and   (\ref{trsAD}), $\Box$: LNGMs of (\ref{trsAD}) and (\ref{trsAB}) .
   %At least   $20\%$  of instances, feaeible region have two LNGM.
}}
\end{figure}
\begin{table}
\centering
\scriptsize
\small
\centering
\begin{tabular}{ c c c c c c c c c c }
\hline
 &    { Hybrid algorithm\hspace*{-2cm}} &        &       &   &{Sakaue et. al algorithm \cite{SakaueNakat:16}\hspace*{-2.7cm}} & &  \\
\cline{1-4}\cline{6-8}
 n & Obj & CPU & KKT &   & Obj & CPU & KKT \\
\hline
5 & -132.94 & 1.52 & 1.23e-08 &   & -132.94 & 0.04 & 1.47e-10  \\
%\hline
 10 & -144.27 & 1.85 & 5.15e-08 &   & -144.27 & 0.57 & 1.73e-11  \\
%\hline
15 & -139.85 & 2.19 & 3.50e-08 &   & -139.85 & 8.75 & 1.74e-10 \\
%\hline
20 & -114.12 & 2.17 & 6.49e-08 &   & -114.12 & 49.02 & 1.16e-09  \\
%\hline
 25 & -114.64 & 2.33 & 1.17e-08 &   & -114.64 & 186.34 & 2.20e-09 \\
%\hline
30 & -134.51 &2.81 & 6.54e-08 &   & -134.51 & 597.72 & 9.57e-09 \\
\hline
\end{tabular}
\caption{ \scriptsize  Comparison with   Sakaue et. al's  algorithm  \cite{SakaueNakat:16} when  $Den= 1$.}
\label{tabjacas2}
\end{table}
\begin{table}
\centering
%\scriptsize
\small
\begin{tabular}{cccccccccccc}
\hline
n& \scriptsize{CPU({Hybrid algorithm})} & \scriptsize{$||F_{Hy}-F_{Sn}||<10^{-6} $} & \scriptsize{$F_{Hy}<F_{Sn} $} & \scriptsize{Cv-Snopt} & \scriptsize{CPU(Snopt)} \\
\hline
 50 & 2.91 & 83     & 10        & 4              & 4.66  \\
%\hline
100& 4.22     & 81     &14         & 2           & 4.62          \\
%\hline%%%%%%%%%%%%%%%%
200 & 10.96   &80      &10         &10          & 15.82  \\
%\hline%%%%%%%%%%%%%%%%%%%%%%%%%%%
300 & 23.62   &86       &12         & 2           & 30.45   \\
%\hline%%%%%%%%%%%%%%%%%%%%%%%%%%%%%%%
500& 62.32     & 82      & 16       & 0           & 70.43  \\
\hline%%%%%%%%%%%%%%%%%%%%
\end{tabular}
\caption{\scriptsize Comparison with Snopt solver of Tomlab when $Den= 1$.\hspace*{-.6cm}}
\label{tabTo1cas2}
\end{table}\begin{table}
\centering
%\scriptsize
\small
\begin{tabular}{ccccccccccc}
\hline
 n & \scriptsize{CPU({Hybrid algorithm})} & \scriptsize{$||F_{Hy}-F_{Sn}||<10^{-6} $} & \scriptsize{$F_{Hy}<F_{Sn} $} & \scriptsize{Cv-Snopt} & \scriptsize{CPU(Snopt)} \\
\hline
100   &  1.89     & 71     &26         & 0           & 3.14          \\
%\hline%%%%%%%%%%%%%%%%
200  & 3.35    &86      &12         &1          & 11.51  \\
%\hline%%%%%%%%%%%%%%%%%%%%%%%%%%%
300 &  3.88   &78       &19         & 3           & 19.63   \\
%\hline%%%%%%%%%%%%%%%%%%%%%%%%%%%%%%%
500& 7.35       & 80      & 19       & 0           & 52.27   \\
%\hline%%%%%%%%%%%%%%%%%%%%%%%%%%%%%
700&  18.88    &  76      &     21   &        3     & 107.36   \\
%\hline%%%%%%%%%%%%%%%%%%%%%%%%%%%
1000 & 30.62 &75        &19       & 3           &  233.92   \\
%\hline%%%%%%%%%%%%%%%%%%%%%%%%%%%%%
2000&241.68  & 67       & 22        &10         & 1429.76    \\
\hline%%%%%%%%%%%%%%%%%%%%
\end{tabular}
\caption{\scriptsize  Comparison with Snopt solver of Tomlab when $Den=0.1$.}
\label{tabTo2cas2}
\end{table}
%%%%%%%%%%%%%%%%%%%%%%%%%%%%%%%%%%%%%%%%%%%%%%
%%%%%%%%%%%%%%%%%%%%%%%%%%%%%%%%%%%%%%%%%%%%%%%%%
\begin{table}
\centering
\small
\begin{tabular}{cccccccccccc}
\hline
n    & CPU     & KKT       & L(\ref{trsAD})  & L(\ref{trsAB})    & Opt-2active  & Opt-L (\ref{trsAD}) &  Opt-L (\ref{trsAB})    \\
\hline
100& 5.76    & 3.11e-08 & 21 & 100   & 3       & 10    & 87      \\
%\hline
 200& 19.87 & 3.32e-08 & 22 & 100   & 1       & 9      & 90    \\
%\hline
 300&28.91  &  7.13e-08 & 29 & 100  & 4       & 11     & 85       \\
%\hline
 400& 86.45 & 4.87e-08 & 32  & 100  & 1       & 12     & 87      \\
%\hline
 500 &136.54& 3.20e-08& 37  & 100  & 3       & 11     & 86       \\
\hline
\end{tabular}
 \caption{\scriptsize Result of  {Hybrid} algorithm for second class of test problem  when $Den=1$.}
\label{tab1cas2}
\end{table}
%%%%%%%%%%%%%%%%%%%%%%%%%%%%%%%%%%%%%%%%%
%%%%%%%%%%%%%%%%%%%%%%%%%%%%%%%%%%%%%%%%%
\item {\bf  ُThird class of test problems}\\
In this class, we generate     \TTRS instances   where    LNGMs
 and global minimizers  of  (\ref{trsAD}) and   (\ref{trsAB}) are all infeasible for \TTRSp.
 This is done in two ways. In the first method, matrix $A$ is generated such that the multiplicity of it's minimum eigenvalue is at least two. Thus  \TRS has no \LNGM.  The second  method is based on Corollary \ref{colng1}.
Moreover, in this class, strong duality holds for  at  least  $90\%$   of the generated instances.
 Starting point for {Hybrid} algorithm uses (\ref{intialpoint}) and
 according to our extensive testing   $\tau = 0.9$, $\rho = 2|\lambda_1(A)| + 1$, and $\lambda=4 x_0 $ are appropriate choices.
  Results are summarized in Tables \ref{case2jasdtable} and \ref{case2tomtable} for the  average of 100 runs.
 In dimensions $5$ to $30$, we compare {Hybrid} algorithm with the    Sakaue et. al's  algorithm  \cite{SakaueNakat:16},
 the corresponding results are reported  in Table \ref{case2jasdtable}.  As we see, our method has   significant advantages over the    Sakaue et. al's  algorithm  \cite{SakaueNakat:16} in term of CPU time while having equal objective values and comparable KKT accuracies.
 In Table \ref{case2tomtable},  we compare the {Hybrid} algorithm with  the Snopt solver in Tomlab. The  Snopt for $n=50~to~300$   is better in term of CPU time but  for larger dimensions {Hybrid} algorithm has better time performance.
 It is worth to note  that the optimal values for both methods are almost the same for most of the problems as shown in the third column of Table \ref{case2tomtable}.
  \begin{table}
\centering
\small
\begin{tabular}{ ccccccccc cc}
\hline
 %{{Hybrid algorithm}}  &         &             &                   &        &{Sakaue  et. al  algorithm  \cite{SakaueNakat:16}}&          &  \\
 &{ Hybrid algorithm\hspace*{-2cm}}     &        &       &   &{Sakaue et. al algorithm \cite{SakaueNakat:16}\hspace*{-2.7cm}} & &\\
\cline{1-4}\cline{6-8}
  n       &Obj    &CPU        & KKT            &        &Obj                                                   & CPU   & KKT             \\
\hline
 5       & -38.31& 0.41       & 6.32e-08    &         &-38.31                                              & 0.03    & 1.92e-10   \\
%\hline
10       &-48.94& 1.06        & 5.81e-08   &         &-48.94                                               & 0.49    & 1.08e-09   \\
%\hline
15       &-57.29 & 1.11        & 5.15e-08    &        &-57.29                                               & 7.77     & 1.73e-11     \\
%\hline
 20       &-68.56& 0.98         & 4.72e-08    &      &-68.56                                             & 46.08   & 1.71e-12      \\
%\hline
25       &-75.38 & 1.06         & 4.43e-08    &        &-75.38                                                & 188.56  & 7.63e-12      \\
%\hline
30       &-16.42 & 1.67         & 4.21e-10    &        &-16.42                                                & 710.52  & 7.28e-15      \\
\hline
\end{tabular}
\caption{ \scriptsize Comparison with    Sakaue et. al's  algorithm  \cite{SakaueNakat:16} when $Den=1$.}
\label{case2jasdtable}
\end{table}
\begin{table}
\centering
\small
\begin{tabular}{cccccccccccc}
\hline
n&Den & \scriptsize{CPU({Hybrid algorithm})} & \scriptsize{$||F_{Hy}-F_{Sn}||<10^{-6} $} & \scriptsize{$F_{Hy}<F_{Sn} $} & \scriptsize{Cv-Snopt} & \scriptsize{CPU(Snopt)} \\
\hline
50  & 1  & 2.96     &96      & 2           & 2          & 6.31       \\
%\hline%%%%%%%%%%%%%%%%%%%%%%%%%%%%%%%%%%%%%%
100 & 1 &6.31       & 97     & 3           & 0          & 5.34     \\
%\hline%%%%%%%%%%%%%%%%%%%%%%%%%%%%%%%%%%
200 & 1 & 22.57    & 88     &10           &2          & 18.33   \\
%\hline%%%%%%%%%%%%%%%%%%%%%%%%%%%%%%%%%%%%%%%%%
300 & 1 & 74.92     & 82     &13           & 5        & 55.05    \\
%\hline%%%%%%%%%%%%%%%%%%%%%%%%%%%%%%%%%%%%%%%
500  &0.1& 6.16    & 98     & 0             & 2          & 45.77  \\
%\hline%%%%%%%%%%%%%%%%%%%%%%%%%%%%%%%%%%%%%%%%%
700 & 0.1 & 7.58     & 98     &0           & 2        & 80.77    \\
%\hline%%%%%%%%%%%%%%%%%%%%%%%%%%%%%%%%%%%%%%%
800 & 0.1 & 9.13     & 95     &2           & 3        & 112.93    \\
%\hline%%%%%%%%%%%%%%%%%%%%%%%%%%%%%%%%%%%%%%%
900 & 0.1 & 10.84     & 95     &1           & 4        & 134.25    \\
%\hline%%%%%%%%%%%%%%%%%%%%%%%%%%%%%%%%%%%%%%%
1000 & 0.1 & 11.07     & 95     &0           & 5        & 178.83    \\
\hline%%%%%%%%%%%%%%%%%%%%%%%%%%%%%%%%%%%%%%%%%%
\end{tabular}
%\centering
\caption{\scriptsize Comparison of  {Hybrid} algorithm with Snopt solver of Tomlab.\hspace*{-1.7cm}}
\label{case2tomtable}
\end{table}
 %%%%%%%%%%%%%%%%%%%%%%%%%%%%%%%%%%%%%%%%%
%%%%%%%%%%%%%%%%%%%%%%%%%%%%%%%%%%%%%%%%%
\\
\item {\bf Forth class of test problems}( Homogeneous problem) \\
For this class, we consider $a=0$ and $c=0$ in \TTRSp.
Consequently, strong duality holds and thus  SDP relaxation of \TTRS is   exact  \cite{YYe:15}.
Since the optimal solution of {Hybrid} algorithm satisfies the conditions of Theorem \ref{globalsemi}, it results in the optimal solution of TTRS.
 In this class, we compare the {Hybrid} algorithm  in terms of CPU time and objective value with  CVX software solving SDP relaxation.
 We set $\tau = 0.9$, $\rho = 4|\lambda_1(A)| + 1$ and  $\lambda=4 x_0 $, where $x_0$ is given by (\ref{intialpoint}). The results are summarized in Table  \ref{case3cvxtable}  for the average of 100 runs. As we see, for $n\le 700$  CVX is faster,  while both have the same optimal objective value. However,
for  $n\geq  1000$, {Hybrid} algorithm solves   the problem to global optimality  while CVX can not be applied.
\begin{table}
\centering
\small
\begin{tabular}{ccccccccc}\hline
&      &   \hspace*{-1cm}     {Hybrid algorithm}  \hspace*{-1.7cm}      &          &                &     &\hspace*{-0.5cm}  SDP   \hspace*{-1.8cm}          &    \\
\cline{1-5}\cline{7-8}
n    & Den   & Obj          & CPU        & KKT              &     & Obj              & CPU  \\
\hline
50  & 1        & -37.93  & 5.41        & 4.9e-08        &       &-37.93        & 0.58   \\
%\hline%%%%%%%%%%%%%%%%%%%%%%%%%%%%%%%%%
100 & 1       & -57.84  & 10.34      & 8.53e-08        &     &-57.84       & 1.16   \\
%\hline%%%%%%%%%%%%%%%%%%%%%%%%%%%
300 & 1      &-100.70  & 95.93      & 1.08e-08        &      & -100.70     & 15.44  \\
%\hline%%%%%%%%%%%%%%%%%%%%%%%%%%%%%%%%%%%
 500 & 0.1   & -42.94 &57.23       & 9.94e-08        &      & -42.94     & 27.16  \\
%\hline%%%%%%%%%%%%%%%%%%%%%%%%%%%%%%%%%
 700 & 0.1   & -49.75 &116.31     & 9.94e-08        &      & -49.75     & 70.58    \\
%\hline%%%%%%%%%%%%%%%%%%%%%%%%%%%%%%%%%
1000 & 0.1 & -56.98     & 241.89    & 2.25e-08      &     & -56.98      & 366.64  \\
%\hline
2000 &0.01 & -28.48     & 162.45    & 5.17e-08      &     & $-$      & $-$  \\
%\hline
3000 &0.001 & -15.39      & 57.35    &1.28e-08      &     & $-$      & $-$  \\
%\hline
4000 &0.001 & -19.29     & 140.63    &6.29e-08      &     & $-$      & $-$  \\
%\hline
5000 &0.001 & -16.21     & 209.17    &1.01e-08      &     & $-$      & $-$  \\
\hline
\end{tabular}
\caption{ \scriptsize  Comparison of  {Hybrid} algorithm with  CVX.}
\label{case3cvxtable}
\end{table}
\end{itemize}

\section{Conclusions}
In this paper, a hybrid algorithm which take advantages of efficient algorithms for finding global and local non-global minimizers of TRS and alternating direction method of multipliers (ADMM) is proposed to tackle the  two-trust-region subproblem.
The convergence of ADMM steps to the first-order stationary condition is proved.
Our numerical experiments on several classes of test problems show that for small-scale problems { hybrid} algorithm has better performance in overall compared to the polynomial-time algorithm of Sakaue et. al's \cite{SakaueNakat:16}.
 Moreover, on medium and large-scale problems comparison with Snopt from Tomlab, as the software giving best results, show that in term of running time,   hybrid algorithm   is better. Also for  large-scale homogeneous problems, hybrid  algorithm outperforms   CVX software.
%
%\section{Acknowledgments} The authors would like to thank both reviewers for their useful comments and questions which improved the paper and University of Guilan for supporting this research. %This research was done when the first  author was on Sabbatical leave at the University of Waterloo, Canada  from the  University of  Guilan, Iran. The author would like to thank the University of Guilan for the financial support and Prof. Henry Wolkowicz for his useful discussion on the subject.

\end{document}